\documentclass{amsart}
\usepackage[utf8]{inputenc}

\usepackage{amsthm, amsmath, amsfonts, amssymb}
\usepackage{mathtools}

\usepackage[shortlabels]{enumitem}
\usepackage[capitalise]{cleveref}
\usepackage[margin=1.5in]{geometry}

\newtheorem{theorem}{Theorem}[section]
\newtheorem{proposition}[theorem]{Proposition}
\newtheorem{lemma}[theorem]{Lemma}
\newtheorem{corollary}[theorem]{Corollary}
\newtheorem{conjecture}[theorem]{Conjecture}

\theoremstyle{definition}
\newtheorem{definition}[theorem]{Definition}
\newtheorem{remark}[theorem]{Remark}

\newcommand{\EE}{\mathbb{E}}
\newcommand{\RR}{\mathbb{R}}

\title{Equiangular lines and large multiplicity of fixed second eigenvalue}
\author{Carl Schildkraut}
\address{ Department of Mathematics, Massachusetts Institute of Technology, Cambridge, MA 02139, USA}
\email{carlsc@mit.edu}

\begin{document}

\maketitle

\begin{abstract}
Answering a question of Jiang and Polyanskii as well as Jiang, Tidor, Yao, Zhang, and Zhao, we show the existence of infinitely many angles $\theta$ for which the maximum number of lines in $\RR^n$ meeting at the origin with pairwise angles $\theta$ exceeds $n+\Omega(\log\log n)$ but is at most $n+o(n)$. To accomplish this, we construct, for various real $\lambda$ and integer $d$, $d$-regular graphs with second eigenvalue exactly $\lambda$ and arbitrarily large second eigenvalue multiplicity. Central to our construction is a distribution on factors of bipartite graphs which possesses concentration properties.
\end{abstract}

\section{Introduction}\label{sec:intro}

A set of lines in $\RR^n$ passing through the origin is called \emph{equiangular} if each pair meets at the same angle $\theta$. The question of the maximum number of equiangular lines in given dimension is natural and well-studied. A short linear-algebraic argument due to Gerzon (see \cite{LemmensSeidel}) gives an upper bound in dimension $n$ of $\binom{n+1}2$, and a lower bound on the order of $n^2$ holds due to a construction of de Caen \cite{deCaen}. 

\vspace{2mm}

One may also consider the setting where the common angle is fixed. For an $\alpha\in(0,1)$, let $N_\alpha(n)$ denote the maximum number of lines in $\RR^n$ each pair of which meets at an angle $\arccos\alpha$. One can see that $N_\alpha(n)\geq n$ for each $\alpha$ and $n$ by constructing the lines spanned by vectors generating the Gram matrix with ones on the diagonal and $\alpha$ off-diagonal. On the other hand, $N_\alpha(n)\leq n+2$ unless $(1-\alpha)/(2\alpha)$ is a totally real algebraic integer which exceeds all of its Galois conjugates \cite[Proposition 23]{JiangPolyanskii}. 

\vspace{2mm}

Much of the literature on $N_\alpha(n)$ focuses on the limit $\lim_{n\to\infty}N_\alpha(n)/n$. Due to a recent result of Jiang, Tidor, Yao, Zhang, and Zhao \cite{JTYZZ} confirming a conjecture of Jiang and Polyanskii \cite[Conjecture A]{JiangPolyanskii}, this limit has (essentially) been completely determined. We begin by stating their result, for which one definition is necessary. Whenever we speak of the eigenvalues of a graph $G$, we refer to the eigenvalues of its adjacency matrix $A_G$. 

\begin{definition}\label{def:spec-radius-order} For a positive real $\lambda$, the \emph{spectral radius order} $k(\lambda)$ of $\lambda$ is the minimum number of vertices of a graph with largest eigenvalue exactly $\lambda$, or $\infty$ if no such graph exists.
\end{definition}

\begin{theorem}[{\cite[Theorem 
 1.2]{JTYZZ}}]\label{thm:JTYZZ} Fix $\alpha\in(0,1)$, and let $\lambda=(1-\alpha)/(2\alpha)$.
\begin{enumerate}[(a)]
    \item If $k=k(\lambda)$ is finite, then $N_\alpha(n)=\lfloor k(n-1)/(k-1)\rfloor$ for all sufficiently large $n$.
    \item If $k(\lambda)=\infty$, then $N_\alpha(n)=n+o(n)$.
\end{enumerate}
\end{theorem}
\noindent The $o(n)$ term in case (b) given by \cite{JTYZZ} is on the order of $n/\log\log n$, and arises from the same bound on the multiplicity of the second largest eigenvalue of the adjacency matrix of a bounded degree graph. This result inspired investigation into bounds on this multiplicity. McKenzie, Rasmussen, and Srivastava \cite{MRS} improved the upper bound to $n/\log^{1/5-o(1)}n$ in the case of regular graphs, and also provided sublinear upper bounds in the unbounded degree case. On the other side, Haiman, Schildkraut, Zhang, and Zhao \cite{HSZZ} gave a lower bound (i.e. a construction of bounded-degree graphs with large second eigenvalue multiplicity) of $n^{1/2-o(1)}$.

\vspace{2mm}

While \cref{thm:JTYZZ} determines the main term of $N_\alpha(n)$ as $n$ grows, it does not establish the order of secondary terms in the $k(\lambda)=\infty$ case. The following bound is conjectured in \cite{JiangPolyanskii} and \cite{JTYZZ}:

\begin{conjecture}[{\cite[Conjecture B]{JiangPolyanskii}, \cite[Conjecture 6.1]{JTYZZ}}]\label{conj:main} Fix $\alpha\in (0,1)$, and let $\lambda=(1-\alpha)/(2\alpha)$. If $k(\lambda)=\infty$, then $N_\alpha(n)=n+O_\alpha(1)$.
\end{conjecture}

Our main result is a disproof of \cref{conj:main} for infinitely many $\alpha$, demonstrating the existence of ``intermediate'' behavior of $N_\alpha(n)$.

\begin{theorem}\label{thm:main} For infinitely many $\alpha\in(0,1)$, the function $N_\alpha(n)$ satisfies 
$$n+\Omega(\log\log n)\leq N_\alpha(n)\leq n+O(n/\log\log n).$$
\end{theorem}

We obtain this result via constructing a family of bounded-degree regular graphs with second eigenvalue exactly $\lambda=(1-\alpha)/(2\alpha)$ for various $\lambda$. The second eigenvalue multiplicity we construct is much smaller than that in \cite{HSZZ}. However, to our knowledge, this is the first result giving unbounded second eigenvalue multiplicity when the eigenvalue is fixed.

\begin{theorem}\label{thm:large-mult-nice} For infinitely many positive real $\lambda$, there exist infinitely many connected $n$-vertex graphs with bounded degree and second eigenvalue exactly $\lambda$ of multiplicity $\Omega(\log\log n)$. In particular, this holds when
\begin{enumerate}[(a)]
    \item $\lambda$ is a sufficiently large integer, or
    \item $\lambda=2\sqrt{u^2+1}-1$ for a sufficiently large integer $u$.
\end{enumerate}
\end{theorem}

In fact, we can prove this result for any real $\lambda$ which is the largest eigenvalue of a symmetric integer matrix satisfying some technical conditions; see \cref{def:matrix-set} and \cref{cor:general-loglog}. A more precise version of part (b) --- \cref{thm:large-mult-nonint} --- will be enough to give \cref{thm:main}.

\section{Proof overview}\label{sec:overview}

In this section, we give an overview of how the graphs satisfying \cref{thm:large-mult-nice} are constructed. Inspired by Marcus, Spielman, and Srivastava \cite{MSS}, the construction will proceed inductively via the \emph{$2$-lift} operation, first studied in conjunction with eigenvalues by Bilu and Linial \cite{BiluLinial}. We construct a sequence $G_0,G_1,\dots$ of graphs so that each $G_i$ has only one eigenvalue greater than $\lambda$ and for which $\lambda$ is an eigenvalue of $G_i$ of multiplicity at least $i$. 

\vspace{2mm}

For any graph $G$, we denote by $A_G$ its \emph{adjacency matrix}, the $V(G)\times V(G)$ matrix where the entry $vw$ for $v,w\in V(G)$ is $1$ if and only if $vw\in E(G)$. For any square matrix $M$ and any positive integer $i$ at most the dimension of $M$, we let $\lambda_i(M)$ denote the $i$th largest eigenvalue of $M$, where eigenvalues are counted with multiplicity.

\subsection*{Integer second eigenvalue} When $\lambda$ is an integer, i.e. to obtain \cref{thm:large-mult-nice}(a), the process is as follows:

\begin{enumerate}
    \item Begin with the bipartite graph $G_0=K_{d,d}$ for some $d$ slightly larger than $\lambda$.
    \item From $G_i$, repeatedly apply $2$-lifts which (i) add no new large eigenvalues and (ii) decrease the $L^\infty$ norm of the existing eigenvectors with large eigenvalue. To construct such lifts, we use a result of Marcus, Spielman, and Srivastava \cite{MSS}; the precise properties we need are stated here in \cref{cor:ramanujan}. Call the graph resulting from these successive lifts $G_i'$.
    \item Once we have performed these lifts, select a subset of edges of $G_i'$ randomly from an appropriate distribution (see \cref{sec:random-factor}) such that each vertex of $G_i'$ is incident to exactly $a:=(d-\lambda)/2$ edges in the subset. This subset informs a choice of $2$-lift $G_{i+1}$ of $G_i'$. For this lift, the vector assigning $1$ to all vertices in one copy of $V(G_i')$ and $-1$ to all vertices in the other copy is an eigenvector of $A_{G_{i+1}}$ with eigenvalue $d-2a=\lambda$. Assuming $d$ is chosen suitably, this is with positive probability the largest new eigenvalue added. 
\end{enumerate}

\subsection*{Non-integer second eigenvalue} To obtain \cref{thm:large-mult-nice}(b), we follow the same general strategy as above, but with some modifications. In step (1), instead of constructing $d$-regular bipartite graphs, we use $k$-partite graphs for which the graph between each pair of parts is a regular bipartite graph. In step (2), we perform the $2$-lifts of Marcus, Spielman, and Srivastava on each pair of parts. Finally, in (3), we choose a subset of edges from between each pair of parts which forms a regular bipartite graph of specified degree. Varying the degrees we choose between parts allows us to vary the eigenvalue which repeats. For a more detailed overview of how the degrees between parts are chosen, see the beginning of \cref{sec:mult-gen}.

\subsection*{$2$-lifts} Given a graph $G$, a \emph{signing} of $G$ is a map $s:E(G)\to\{\pm 1\}$. Each signing possesses an adjacency matrix $A_s$ where entries of $A_G$ corresponding to an edge $e$ are replaced with $s(e)$. To any signing $s$ of $G$, one may associate a graph $\tilde G$ on the doubled vertex set $\{v_0,v_1\colon v\in V(G)\}$ wherein, for each edge $uv\in E(G)$, edges $u_0v_0$ and $u_1v_1$ are drawn in $\tilde G$ if $s(uv)=1$, and edges $u_0v_1$ and $u_1v_0$ are drawn in $\tilde G$ if $s(uv)=-1$. The adjacency matrix of $\tilde G$ can be written in block form as
$$A_{\tilde G}=\frac 12\begin{pmatrix}A_G+A_s&A_G-A_s\\A_G-A_s&A_G+A_s\end{pmatrix}.$$
Such a graph $\tilde G$ is called a \emph{$2$-lift} of $G$, on account of the $2$-covering map $\tilde G\to G$ mapping $v_0$ and $v_1$ to $v$. We make the following observations:
\begin{enumerate}
    \item The multiset of the eigenvalues of $A_{\tilde G}$ is the disjoint union of that of $A_G$ and that of $A_s$. If $x$ is an eigenvector of $A_G$ (resp. $A_s$) with eigenvalue $\mu$, then $x\oplus x$ (resp. $x\oplus -x$) is an eigenvector of $A_{\tilde G}$ with the same eigenvalue $\mu$.
    \item If a vertex $v\in V(G)$ has degree $d_v$, the vertices $v_0,v_1\in V(\tilde G)$ above $v$ also have degree $d_v$. 
    \item If $G$ is $k$-partite, with $V(G)=V_1\sqcup V_2\sqcup\cdots\sqcup V_k$, then $\tilde G$ is also $k$-partite; there are no edges between vertices corresponding to elements of $V_i$ for each $i$. 
\end{enumerate}

\subsection*{Ramanujan lifts} A special case of a theorem of Marcus, Spielman, and Srivastava gives the existence of lifts which, in a very strong sense, add no large eigenvalues.

\begin{theorem}[{\cite[Theorem 5.3]{MSS}}]\label{thm:ramanujan} Every $d$-regular graph possesses a signing whose adjacency matrix $A_s$ has second eigenvalue at most $2\sqrt{d-1}$.
\end{theorem}

\noindent We will make use of the following restatement in the case of bipartite graphs.

\begin{corollary}\label{cor:ramanujan} For every $d$-regular bipartite graph $G$, there exists a $2$-lift $\tilde G$ of $G$ for which, for every eigenvalue $\mu$ of $A_{\tilde G}$ satisfying $|\mu|>2\sqrt{d-1}$, the $\mu$-eigenspace of $A_{\tilde G}$ has the same dimension as that of $A_G$. Moreover, if $\{x_1,\dots,x_k\}$ is a basis of the $\mu$-eigenspace of $A_G$, then $\{x_1\oplus x_1,\dots,x_k\oplus x_k\}$ is a basis of the $\mu$-eigenspace of $A_G$.
\end{corollary}
\begin{proof} Let $s$ be a signing of $G$ as in \cref{thm:ramanujan} so that $\lambda_1(A_s)\leq 2\sqrt{d-1}$. Since $A_s$ is a signing of a bipartite graph, its spectrum is symmetric about $0$, and so $\|A_s\|\leq 2\sqrt{d-1}$. The statement then follows from observation (1) on $2$-lifts.
\end{proof}

\subsection*{Multiplicity-incrementing lifts and graph factors} While we need Ramanujan lifts to ensure we don't have too many large eigenvalues, we will construct more specialized $2$-lifts for the step in which second eigenvalue multiplicity is incremented. Such lifts come from choosing appropriate \emph{$a$-factors}.

\begin{definition}\label{def:factor} Given a $d$-regular graph $G$ and an integer $0\leq a\leq d$, an \emph{$a$-factor} of $G$ is an $a$-regular subgraph of $G$ on the vertex set $V(G)$. Equivalently, $a$-factors are determined by selection of a subset $S\subset E(G)$ in which every vertex of $G$ is incident to exactly $a$ edges in $S$.
\end{definition}

As stated in the outline of our approach (specialized to integer eigenvalue), we shall lift by a signing in which each vertex is incident to the same number of edges signed $-1$. Letting $H$ denote the graph consisting of the edges between two parts which are signed $-1$, this is equivalent to $H$ being a factor of $G$. Consider choosing an $a$-factor $H$ of a $d$-regular graph $G$ at random from a distribution in which each edge appears in $H$ with equal probability, and let $s$ be the signing determined by letting $s(e)=-1$ if and only if $e\in E(H)$. Then
$$\EE[A_s]=\EE[A_G-2A_H]=\left(1-\frac{2a}d\right)A_G.$$
Given any $H$, $A_s$ will have an eigenvector (the all-ones vector) with eigenvalue $d-2a$. On the other hand, if we can choose $H$ so that $A_s$ is close in spectral norm to its expectation, then the eigenvalues of $A_G$ should ``shrink'' by a factor of $1-2a/d$ up to some small error, and $d-2a$ should be the largest eigenvalue of $A_s$. We are not able to show that $H$ can be chosen in this way. However, in accordance with this goal, we give the following proposition, which gives the existence of a distribution on $a$-factors of a $d$-regular bipartite graph $G$ that, up to a matrix of small spectral norm, displays concentration for Lipschitz functions of the form $u^\intercal A_Hv$. Once we have performed enough Ramanujan lifts, this concentration result will suffices to show the existence of a signing $s$ with $\lambda_1(A_s)=d-2a$.

\begin{proposition}\label{prop:concentration-general} Let $a$ and $d$ be positive integers with $0\leq a\leq d$. Let $G$ be a bipartite $d$-regular graph on $n$ vertices. There exists a probability distribution on the set of pairs $(H,M)$ where $H$ is an $a$-factor of $G$ and $M$ is an $n\times n$ matrix with $\|M\|\leq 6\sqrt d$ such that, for any vectors $u,v\in\RR^n$ and any real $t\geq 0$,
$$\Pr\left[\left|u^\intercal\left(A_H-\frac adA_G+M\right)v\right|>t(2+\sqrt 2)\sqrt{dn\lceil\log_2 n\rceil}\|u\|_\infty\|v\|_\infty\right]\leq 2(\log_2 d)e^{-t^2/2}.$$
\end{proposition}

\noindent The matrix $M$ is more an artifact of the proof than an essential component of the distribution, and one should think of the distribution described in the proposition as a distribution solely on $a$-factors with some ``extra data'' attached which is useful for analysis.

\vspace{2mm}

We conclude this section with an outline of the remainder of the paper. In \cref{sec:random-factor}, we show \cref{prop:concentration-general} and give a corollary which allows us to construct $a$-factors whose adjacency matrices are, on particular subspaces of $\RR^{V(G)}$, close to $(a/d)A_G$ in spectral norm. In \cref{sec:mult-int}, we apply this construction to show case (a) of \cref{thm:large-mult-nice}, giving large multiplicity of large integer second eigenvalues from regular bipartite graphs. In \cref{sec:mult-gen}, we generalize this construction to $k$-partite graphs and show case (b) of \cref{thm:large-mult-nice}. Then, in \cref{sec:equi}, we apply the graphs constructed in \cref{thm:large-mult-nice}(b) to the problem of equiangular lines, giving \cref{thm:main}.

\section{Random factors of graphs}\label{sec:random-factor}

In this section, we show \cref{prop:concentration-general}, which gives a distribution on $a$-factors $H$ of any $d$-regular bipartite graph $G$ which, up to a matrix of small spectral norm, displays relatively tight concentration. The distribution is constructed as follows:

\begin{enumerate}
    \item In the special case when $d$ is even and $a=d/2$:
    \begin{enumerate}
        \item Partition $G$ into edge-disjoint cycles, at each step adding the smallest remaining cycle to the partition. \cref{lem:no-small-cycles-eig} shows that the union of the large cycles has small spectral norm.
        \item Note that all cycles in the partition have even length. From each cycle with $2\ell$ edges, choose an alternating set of $\ell$ edges uniformly at random from the two possible choices and add this set to $H$. For large cycles, one may make the choices deterministically. For example, if one such cycle is $v_1v_2v_3v_4v_5v_6$, add either the edges $\{v_1v_2,v_3v_4,v_5v_6\}$ or the edges $\{v_2v_3,v_4v_5,v_6v_1\}$ to $H$.
    \end{enumerate}
    \item Using (1) as a subroutine, proceed recursively:
    \begin{enumerate}
        \item In the base case of $d=1$, there is nothing to choose: $G$ has a unique $0$-factor and a unique $1$-factor.
        \item If $a>d/2$, construct a random $(d-a)$-factor and take its complement in $G$.
        \item If $a\leq d/2$ and $d$ is even, construct a random $(d/2)$-factor $G'$ of $G$ using (1), and then construct a random $a$-factor of the $d/2$-regular bipartite graph $G'$.
        \item If $d$ is odd, remove a $1$-factor (i.e. a perfect matching) from $G$ arbitrarily to form a $(d-1)$-regular graph $G'$, and then find a random $a$-factor of $G'$.
    \end{enumerate}
\end{enumerate}
\noindent The terms arising from the large cycles in step 1(b) and from the arbitrary perfect matchings in 2(d) are absorbed into the auxiliary matrix $M$. What is left is a sum of many ``small'' random choices, and so concentration arises from Chernoff-style arguments.

\vspace{2mm}

To prove our concentration result, we will first require the following technical lemma, which assists in bounding the spectral radius of $M$.

\begin{lemma}\label{lem:no-small-cycles-eig} Let $G$ be a graph on $n$ vertices with maximum degree at most $d$ and no cycles of length at most $2\lceil\log_2n\rceil$. Then $\|A_G\|\leq (2\sqrt 2)\sqrt d$. 
\end{lemma}
\begin{proof} Let $r=\lceil \log_2n\rceil$. On one hand, we have
$$\operatorname{tr} A_G^{2r}=\sum_{i=1}^n\lambda_i(A_G)^{2r}\geq \lambda_1(A_G)^{2r}.$$
On the other hand, $\operatorname{tr}A_G^{2r}$ counts the number of closed walks of length $2r$ in $G$. Each such walk is also a walk on the universal cover of $G$, since otherwise would imply the existence of a cycle of $G$ of length at most $2r$. Since every vertex has maximum degree $d$, the number of such walks starting at a given vertex is at most $d^rC_r\leq (4d)^r$, where $C_r$ is a Catalan number. This gives
$$n(4d)^r\geq \lambda_1(A_G)^{2r},$$
and so
$$\lambda_1(A_G)\leq n^{1/(2r)}\sqrt{4d}\leq n^{\frac1{2\log_2 n}}\sqrt{4d}=(2\sqrt2)\sqrt d.\eqno\qedhere$$
\end{proof}

We now state and prove the concentration result for step (1) of the above outline.

\begin{lemma}\label{lem:concentration-special} Let $d$ be an even positive integer. Let $G$ be a bipartite $d$-regular graph on $n$ vertices. There exists a (deterministic) $n\times n$ matrix $M$ satisfying $\|M\|\leq \sqrt{2d}$ and a probability distribution on the set of $d/2$-factors $H$ of $G$ such that, for any vectors $u,v\in\RR^n$ and any real $t\geq 0$,
$$\Pr\left[\left|u^\intercal\left(A_H-\frac 12A_G+M\right)v\right|\geq t\sqrt{dn\lceil\log_2 n\rceil}\|u\|_\infty\|v\|_\infty\right]\leq 2e^{-t^2/2}.$$
\end{lemma}
\begin{proof} First, partition the edges $E=E(G)$ of $G$ into sets $E_0$ and $E_1$ as follows: while there is a cycle in $G$ of length at most $2\lceil \log_2n\rceil$, add the edges of this cycle to $E_1$ and remove them from $G$; once no such cycle exists, add all remaining edges to $E_0$. Let $G_0$ and $G_1$ be the graphs on the same vertex set as $G$ with edge sets $E_0$ and $E_1$, respectively. By \cref{lem:no-small-cycles-eig}, $\|A_{G_0}\|\leq 2\sqrt{2d}$. 

\vspace{2mm}

We construct $H$ by partitioning $E_0$ deterministically and $E_1$ randomly. Since $E_1$ is a union of cycles, every vertex has even degree in $G_1$ and thus in $G_0$, and so the edges of $E_0$ can be partitioned into cycles. For each such cycle, pick an arbitrary alternating set of edges of the cycle as in step 1(b) above; let the resulting edges be $E_0'$ and let $H_0$ be the graph with edge set $E_0'$. Each vertex has degree in $H_0$ half of that in $G_0$. Let $M=1/2\cdot A_{G_0}-A_{H_0}$; since $M$ is half the adjacency matrix of a signing of $G_0$, $\|M\|\leq (1/2)\|A_{G_0}\|\leq \sqrt{2d}$. 

\vspace{2mm}

Now, $G_1$ is a union of cycles of length at most $2\lceil \log_2n\rceil$. For each such cycle, choose an alternating set of half the edges randomly from the two possible choices, and add its edges to $E_1'$; let $H_1$ be the graph with edge set $E_1'$. Let $H$ be the graph with edge set $E_0'\cup E_1'$; $H$ is an $a$-factor of $G$, and 
$$A_H-\frac adA_G+M=A_{H_0}+A_{H_1}-\frac12A_{G_0}-\frac12A_{G_1}+\frac12A_{G_0}-A_{H_0}=A_{H_1}-\frac12A_{G_1}.$$
Consider the random variable
$$X=u^\intercal\left(A_{H_1}-\frac12A_{G_1}\right)v.$$
Letting $\mathcal C$ be the set of cycles the union of which is $E_1$,
$$X=\sum_{C\in\mathcal C}X_Cu^\intercal\left(\frac12\tilde A_C\right)v,$$
where $X_C\sim\operatorname{Unif}(\{\pm 1\})$ and $\tilde A_C$ is the adjacency matrix of a signing of $C$ in which incident edges are signed alternately. Since $u^\intercal\tilde A_Cv$ is a sum of $2\operatorname{len}(C)$ terms of the form $\pm u_iv_j$ for integers $i$ and $j$, 
$$\left|u^\intercal \left(\frac 12\tilde A_C\right)v\right|\leq \operatorname{len}(C)\|u\|_\infty\|v\|_\infty.$$
So, the bounded differences inequality implies, for any $\mu\geq 0$,
$$\Pr\left[|X|\geq \mu\right]\leq 2\exp\left(\frac{-2\mu^2}{\sum_{C\in\mathcal C}\left(2\operatorname{len}(C)\|u\|_\infty\|v\|_\infty\right)^2}\right).$$
We have
$$\sum_{C\in\mathcal C}\operatorname{len}(C)^2\leq \left(\max_{C\in\mathcal C}\operatorname{len}(C)\right)\sum_{C\in\mathcal C}\operatorname{len}(C)\leq 2\lceil\log_2n\rceil E(G)=dn\lceil \log_2n\rceil,$$
so
$$\Pr\left[|X|\geq \mu\right]\leq 2\exp\left(\frac{-\mu^2}{2dn\lceil\log_2n\rceil\|u\|_\infty^2\|v\|_\infty^2}\right).$$
This implies the desired concentration upon setting $\mu=t\sqrt{dn\lceil\log_2n\rceil}\|u\|_\infty\|v\|_\infty$. 
\end{proof}

We use this lemma as a black box, as described in step (2) of the outline, to show \cref{prop:concentration-general}.

\begin{proof}[Proof of \cref{prop:concentration-general}] Let $c_d$ be the sequence defined recursively by $c_1=0$, $c_d=c_{d-1}+1$ if $d>1$ is odd, and $c_d=c_{d/2}+\sqrt{2d}$ if $d$ is even. We proceed by strong induction on $d$ to show that there exists a distribution on pairs $(H,M)$ of $a$-factors of $G$ and $n\times n$ matrices wherein $\|M\|\leq c_d$ and, for any $t\geq 0$ and $u,v\in\RR^n$,
$$\Pr\left[\left|u^\intercal\left(A_H-\frac adA_G+M\right)v\right|>t(2+\sqrt 2)\sqrt{dn\lceil\log_2 n\rceil}\|u\|_\infty\|v\|_\infty\right]\leq 2(\log_2 d)e^{-t^2/2}.$$
In the base case of $d=1$, there is not much choice; if $a=0$ then $H$ is empty, while if $a=1$ then $H=G$. In either case, we can take $M=0$, and $u^\intercal(A_H-a/d\cdot A_G+M)v=0$ always. 

\vspace{2mm}

For the inductive step, we first treat the case where $d$ is even and $a\leq d/2$. In this case, we proceed in two steps:
\begin{enumerate}
    \item Construct a $d/2$-factor $G_1$ of $G$ via \cref{lem:concentration-special}, and let $M_1$ be the associated matrix.
    \item Construct an $a$-factor $H$ of $G_1$ via the inductive hypothesis, and let $M_2$ be the associated matrix.
\end{enumerate}
We set $M=M_2+2a/d\cdot M_1$. We have $\|M_1\|\leq \sqrt{2d}$ by \cref{lem:concentration-special} and $\|M_2\|\leq c_{d/2}$ by the inductive hypothesis. This gives
$$\|M\|\leq \|M_2\|+(2a/d)\|M_1\|\leq c_{d/2}+1\cdot \sqrt{2d}=c_d$$
and
$$A_H-\frac adA_G+M=\left(A_H-\frac a{d/2}A_{G_1}+M_2\right)+\frac{2a}d\left( A_{G_1}-\frac 12A_G+M_1\right),$$
as well as the concentration inequalities
\begin{align*}
\Pr\left[\left|u^\intercal\left(A_{G_1}-\frac 12A_G+M_1\right)v\right|\geq t\sqrt{dn\lceil\log_2 n\rceil}\|u\|_\infty\|v\|_\infty\right]&\leq 2e^{-t^2/2}\\
\Pr\left[\begin{matrix}\left|u^\intercal\left(A_H-\frac a{d/2}A_{G_1}+M_2\right)v\right| 
\ \ \ \ \ \ \ \ \ \ \ \ \ \ \ \ \ \ \ \ \\\ \ \ \ \ \ 
\geq t(2+\sqrt 2)\sqrt{(d/2)n\lceil\log_2 n\rceil}\|u\|_\infty\|v\|_\infty\end{matrix}\right]&\leq 2(\log_2(d/2))e^{-t^2/2}.
\end{align*}
This implies by the union bound that the probability that $u^\intercal (A_H-a/d\cdot A_G+M)v$ exceeds
$$t(2+\sqrt 2)\sqrt{(d/2)n\lceil\log_2 n\rceil}\|u\|_\infty\|v\|_\infty+\frac{2a}d\left(t\sqrt{dn\lceil\log_2 n\rceil}\|u\|_\infty\|v\|_\infty\right)$$
in magnitude is at most $2(\log_2d)e^{-t^2/2}$, as this would imply that one of the two above events occurs. Since
$$(2+\sqrt2)\sqrt{d/2}+\frac{2a}d\sqrt d\leq (\sqrt2+1)\sqrt d+\sqrt d=(2+\sqrt2)\sqrt d,$$
this is sufficient.

\vspace{2mm}

In the case where $d$ is even and $a>d/2$, we first construct a pair $(H',M')$ where $H'$ is a $(d-a)$-factor of $G$. Let $H$ be the complement of $H'$ in $G$, and let $M=-M'$. Since
$$A_H-\frac adA_G+M=(A_G-A_{H'})-\frac adA_G-M'=-\left(A_{H'}-\frac{d-a}dA_G+M'\right),$$
the hypothesis for $(d-a)$-factors of $d$-regular graphs suffices. 

\vspace{2mm}

Finally, we treat the case where $d$ is odd. If $a>d/2$ then we can proceed as in the previous paragraph from the $a<d/2$ case, so assume $a<d/2$. We proceed as follows:
\begin{enumerate}
    \item Select an arbitrary perfect matching $H'\subset G$, and let $G_1$ be the complement of $H'$ in $G$, a $(d-1)$-factor of $G$.
    \item Construct an $a$-factor $H$ of $G_1$ via the inductive hypothesis, and let $M_1$ be the associated matrix.
\end{enumerate}
Define $M=\frac{a}dA_{H'}-\frac a{d(d-1)}A_{G_1}+M_1$. We have
\begin{align*}
A_H-\frac adA_G+M
&=A_H-\frac adA_G+\frac adA_{H'}-\frac a{d(d-1)}A_{G_1}+M_1\\
&=A_H-\frac adA_{G_1}-\frac a{d(d-1)}A_{G_1}+M_1\\
&=A_H-\frac a{d-1}A_{G_1}+M_1,
\end{align*}
as well as
$$\|M\|\leq \frac ad\|A_{H'}\|+\frac a{d(d-1)}\|A_{G_1}\|+\|M_1\|\leq \frac ad+\frac ad+c_{d-1}\leq c_{d-1}+1$$
so the inductive hypothesis at $d-1$ suffices.

\vspace{2mm}

All that remains is to show that $c_d\leq 6\sqrt d$ for all $d$. We can show
\begin{align*}
c_d
&\leq (2+\sqrt 2)(\sqrt{2d}-1)+(\#\text{ of ones in the binary representation of }d)\\
&\leq (2+\sqrt 2)(\sqrt{2d}-1)+\log_2d
\end{align*}
by induction on $d$, so the result follows from
$$\log_2d-(2+\sqrt 2)\leq (4-2\sqrt 2)\sqrt d,$$
which is easy to check.
\end{proof}

We conclude this section with a corollary of \cref{prop:concentration-general} which shows a concentration result for factors of regular bipartite graphs on subspaces with an orthonormal basis with small $L^\infty$ norm. This is what we will use when constructing our multiplicity-incrementing lifts.

\begin{corollary}\label{cor:factor-selection} Let $a$ and $d$ be positive integers with $0\leq a\leq d$. Let $G$ be a bipartite $d$-regular graph on $n$ vertices, and let $1\leq K,m\leq n$. Let $S$ be a finite set containing $V(G)$, and let $U\subset\RR^S$ be a subspace of dimension at most $m$ spanned by pairwise orthogonal unit vectors $u_1,u_2,\dots,u_{m'}$ satisfying $\|u_i\|_\infty\leq\sqrt{K/n}$ for each $1\leq i\leq m'$. If $m\geq\max(2\log_2d,15)$ and $n\geq K^3m^6$, then there exists an $a$-regular subgraph $H\subset G$ on $V(G)$ satisfying
$$\left|v^\intercal\left(A_H-\frac adA_G\right)v\right|\leq 7\sqrt d$$
for every unit vector $v\in U$, where $A_H$ and $A_G$ are extended by zeros to a matrix in $\RR^{S\times S}$.
\end{corollary}
\begin{proof} For each $1\leq i\leq m'$, let $\overline{u_i}$ be the restriction of $u_i$ to $\RR^{V(G)}$; note that $\|\overline{u_i}\|_\infty\leq\|u_i\|_\infty$. Set $t=\sqrt{6.5\ln m}$, so that
$$2(\log_2d)e^{-t^2/2}\leq me^{-3.25\ln m}<m^{-2}.$$
By \cref{prop:concentration-general} and the union bound, there exists an $a$-regular subgraph $H\subset G$ on $V(G)$ and an $n\times n$ matrix $M$ satisfying $\|M\|\leq 4\sqrt d$ for which
\begin{align*}
\left|u_i^\intercal\left(A_H-\frac adA_G+M\right)u_j\right|
&=\left|\overline{u_i}^\intercal\left(A_H-\frac adA_G+M\right)\overline{u_j}\right|\\
&\leq t(2+\sqrt 2)\sqrt{dn\lceil \log_2 n\rceil}\|u_i\|_\infty\|u_j\|_\infty\\
&<9K\sqrt{\frac{d\ln m\lceil \log_2 n\rceil}n}
\end{align*}
for all $1\leq i,j\leq m'$. Now, take any unit vector $v\in U$. Since $\{u_i\}_i$ is an orthonormal basis of $U$, we can write $v=\sum_{i=1}^{m'} c_iu_i$ for some real $c_1,\dots,c_{m'}$ with $\sum c_i^2=1$. This gives
\begin{align*}
\left|v^\intercal\left(A_H-\frac adA_G+M\right)v\right|
&\leq \sum_{i=1}^{m'}\sum_{j=1}^{m'}|c_i||c_j|\left|u_i^\intercal\left(A_H-\frac adA_G+M\right)u_j\right|\\
&<9K(|c_1|+\cdots+|c_{m'}|)^2\sqrt{\frac{d\ln m\lceil \log_2 n\rceil}n}\\
&\leq 9Km\sqrt{\frac{d\ln m\lceil \log_2 n\rceil}n},
\end{align*}
where we have used the Cauchy--Schwarz inequality. Since $n\geq 2$, $n/\lceil\log_2n\rceil\geq n^{2/3}$, and since $m\geq 15$, $9\sqrt{\ln m}\leq m$, so this bound gives
$$\left|v^\intercal\left(A_H-\frac adA_G+M\right)v\right|\leq Km^2n^{-1/3}\sqrt d\leq \sqrt d$$
by our bound on $n$. The fact that $\|M\|\leq 6\sqrt d$ finishes the proof.
\end{proof}

\section{Integer second eigenvalue}\label{sec:mult-int}

In this section, we use \cref{cor:factor-selection} to construct fixed-degree regular graphs with a fixed second eigenvalue of multiplicity $\Omega(\log\log n)$; the fixed eigenvalue can be any sufficiently large integer. The main result of this section will be the following ``incrementing'' proposition. Iterating this will allow us to obtain \cref{thm:large-mult-nice}(a).

\begin{proposition}\label{prop:big-aug-step} Let $d\geq 22000$, and let $a=\lceil 12\sqrt d\rceil$. Suppose $G$ is a bipartite $d$-regular graph on $n$ vertices for which $\lambda_2(A_G)\leq d-2a$, and let $\ell$ be the multiplicity of $d-2a$ as an eigenvalue of $A_G$. Then there exists a bipartite $d$-regular graph $\tilde G$ on at most $4n^9$ vertices with second eigenvalue $d-2a$ of multiplicity at least $\ell+1$. 
\end{proposition}

The process will be, as described at the beginning of \cref{sec:overview}, to first perform many lifts which add no new large eigenvalue, and then to apply a multiplicity-incrementing lift determined by a random $a$-factor of $G$ using \cref{cor:factor-selection}. The properties of this lift are described in the following lemma.

\begin{lemma}\label{lem:aug-step} Let $d\geq 22000$, and let $a=\lceil 12\sqrt d\rceil$. Let $G$ be a $d$-regular bipartite graph with $n$ vertices satisfying $\lambda_2(A_G)\leq d-2a$. Suppose that there are at most $m$ eigenvalues of $G$ greater than $2\sqrt{d-1}$, and $A_G$ possesses an orthonormal eigenbasis in which each component of these $m'\leq m$ top eigenvectors is at most $\sqrt{m/n}$ in magnitude. Then, as long as $m\geq 2\log_2d$ and $n>m^9$, there exists an $a$-factor $H$ of $G$ with $\lambda_1(A_G-2A_H)=d-2a$.
\end{lemma}

\noindent We first prove the proposition given this lemma, and then prove the lemma.

\begin{proof}[Proof of \cref{prop:big-aug-step}] Let $t$ be a positive integer. By applying  \cref{cor:ramanujan} $t$ times, we can find a $d$-regular bipartite graph $G'$ on $n':=n2^t$ vertices for which, for each eigenvalue $\mu$ of $A_{G'}$ with $|\mu|\geq 2\sqrt{d-1}$, the $\mu$-eigenspace of $A_{G'}$ is exactly
$$\left\{\underbrace{x\oplus x\oplus \cdots\oplus x}_{2^t\text{ copies of }x}\colon A_Gx=\mu x\right\}.$$
In particular, $$\lambda_2(A_{G'})\leq\max(\lambda_2(A_G),2\sqrt{d-1})\leq d-2a,$$
and $A_{G'}$ has eigenvalue $d-2a$ with multiplicity $\ell$. Moreover, $A_G'$ possesses at most $|V(G)|=n$ eigenvalues greater than $2\sqrt{d-1}$, and an orthonormal eigenbasis of $A_{G'}$ containing
$$\left\{\frac1{2^{t/2}}(x\oplus \cdots\oplus x)\colon x\text{ is in an orthonormal eigenbasis of }A_G\right\}$$
possesses no eigenvector with an eigenvalue larger than $2\sqrt{d-1}$ and an entry larger than $2^{-t/2}=\sqrt{n/n'}$ in magnitude. So, as long as $n\geq 2\log_2d$ (which holds since there exists a $d$-regular graph $G$ on $n$ vertices) and $n2^t>n^9$, we can apply \cref{lem:aug-step} to find an $a$-factor $H$ of $G'$ with $\lambda_1(A_{G'}-2A_H)=d-2a$. Letting $s$ be the associated signing of $G'$ and $\tilde G$ be the associated $2$-lift, the spectrum of $\tilde G$ contains the eigenvalue $d-2a$ with multiplicity at least $\ell+1$. Selecting $t$ to be the smallest positive integer satisfying $2^t>n^8$ finishes the proof.
\end{proof}

\begin{proof}[Proof of \cref{lem:aug-step}] Let $U$ be the span of all eigenvectors of $A_G$ with eigenvalue exceeding $2\sqrt{d-1}$ and strictly less than $d$; note that $\dim U\leq m'\leq m$, and $U$ possesses an orthonormal basis in which each vector has $L^\infty$ norm at most $\sqrt{m/n}$. So, by \cref{cor:factor-selection} (which we can apply since $2\log_2d>15$), there exists an $a$-factor $H$ of $G$ with
$$\left|v^\intercal\left(A_H-\frac adA_G\right)v\right|\leq 7\sqrt d$$
for every unit vector $v\in U$. In particular,
\begin{equation}\label{eq:U-bound}\tag{$\star$}
v^\intercal A_Hv\geq \frac adv^\intercal A_Gv-7\sqrt d\|v\|^2\ \ \ \ \ \ \text{for all }v\in U.
\end{equation}

\vspace{2mm}

Now, since $H$ is $a$-regular, the matrix $A_G-2A_H$ has eigenvalue $d-2a$ with the all-ones vector as an eigenvalue. We need to show that there are no larger eigenvalues. To this end, consider any vector $w$ orthogonal to the all-ones vector; it suffices to show $w^\intercal A_Hw\leq d-2a$. Write $w=x+y$, where $x\in U$ and $y\in U'$, the span of all eigenvectors of $A_G$ with eigenvalue at most $2\sqrt{d-1}$. We compute
\begin{align*}
w^\intercal A_Hw
&=x^\intercal A_Hx+2x^\intercal A_Hy+2y^\intercal A_Hy\\
&\geq \frac adx^\intercal A_Gx-(7\sqrt d)\|x\|^2-2\|x\|(ay)-a\|y\|^2
\end{align*}
where we have used \eqref{eq:U-bound} and the fact that $\|A_H\|=a$. Now, this allows us to bound
\begin{align*}
w^\intercal (A_G-2A_H)w
&=x^\intercal A_Gx+y^\intercal A_Gy-2w^\intercal A_Hw\\
&\leq x^\intercal A_Gx+(2\sqrt{d-1})\|y\|^2-\frac{2a}dx^\intercal A_Gx\\
&\ \ \ \ \ \ \ \ \ \ \ \ \ \ \ \ \ \ \ \ \ \ \ \ \ \ \ +(14\sqrt d)\|x\|^2+2a\|y\|(2\|x\|+\|y\|)\\
&=\left(1-\frac{2a}d\right)x^\intercal A_Gx+(14\sqrt d)\|x\|^2+4a\|x\|\|y\|+(2a+2\sqrt{d-1})\|y\|^2\\
&\leq \left((d-2a)\left(1-\frac{2a}d\right)+14\sqrt d\right)\|x\|^2+4a\|x\|\|y\|+(2a+2\sqrt{d-1})\|y\|^2.
\end{align*}
This is a quadratic form in $\|x\|$ and $\|y\|$. The maximum of a quadratic form $Rx^2+2Sxy+Ty^2$ over the unit circle can be easily computed to be $(R+T+\sqrt{(R-T)^2+4S^2})/2$. Since $\|x\|^2+\|y\|^2=\|w\|^2=1$, all that remains is to show that this is at most $d-2a$ for the quadratic form above in $\|x\|$ and $\|y\|$ as long as $d\geq 22000$. This is a simple computation. 
\end{proof}

We conclude this section by establishing \cref{thm:large-mult-nice}(a).

\begin{proof}[Proof of \cref{thm:large-mult-nice}(a)] Let $\ell\geq 20000$ be a positive integer. We will show that, for infinitely many $n$, there exist bounded-degree regular graphs on $n$ vertices with second eigenvalue exactly $\ell$, of multiplicity $\Omega(\log\log n)$.

\vspace{2mm}

Let $d$ be a positive integer exceeding $22000$ for which $d-2\lceil 12\sqrt d\rceil=\ell$ and let $a=\lceil 12\sqrt d\rceil$, so that $\ell=d-2a$. Set $G_0=K_{d,d}$ and $n_0=2d$, and define a sequence $(n_i)$ recursively by $n_{i+1}=4n_i^9$. Note that $\lambda_2(G_0)=0<d-2a$. Repeated applications of \cref{prop:big-aug-step} show that, for each $i\geq 1$, there exists a bipartite $d$-regular graph $G_i$ such that
\begin{enumerate}[(i)]
    \item $G_i$ has at most $n_i$ vertices, and
    \item $d-2a$ is the second eigenvalue of $G_i$ and has multiplicity at least $i$.
\end{enumerate}
The fact that $n_i=2^{2^{\Theta(i)}}$ finishes the proof, since $i=\Theta(\log\log n_i)$.
\end{proof} 

\begin{remark}\label{rmk:factor-improvement} If we were able to find, for any bipartite $d$-regular graph $G$, an $a$-factor $H$ with
$$\left\|A_H-\frac adA_G\right\|=o(a),$$
this would allow us to remove the Ramanujan lift step and attain second eigenvalue multiplicity on the order of $\log n$. This aim is reminiscent of a result of Bilu and Linial \cite[Theorem 3.1]{BiluLinial} that there exists a signing $s$ of such a graph with $\|A_s\|=O(d^{1/2}\log^{3/2}d)$, but in our distribution there seems not to be enough randomness to follow their spectral radius-bounding framework.
\end{remark}

\section{Non-integer second eigenvalue}\label{sec:mult-gen}

In this section we give a generalization of the framework from the previous section, which we will use to prove our main result. We begin by elaborating on the sketch of the $k$-partite graphs-based strategy given in \cref{sec:overview}. To do this, we first need some terminology.

\begin{definition}\label{def:sign-matrix} Recall that a $k\times k$ symmetric matrix $M$ is \emph{irreducible} if the (not necessarily simple) graph with adjacency matrix $M$ is connected. (Irreducibility ensures that we can apply the Perron--Frobenius theorem to $M$.) Given a symmetric $k\times k$ irreducible matrix $M$ with zeros on the diagonal and nonnegative integer entries:
\begin{itemize}
    \item A \emph{graph lift} of $M$ is a lift of the non-simple graph with adjacency matrix $M$. In other words, it is a $k$-partite graph $G$ on vertex set $V_1\sqcup V_2\sqcup\cdots\sqcup V_k$ such that, for each $1\leq i,j\leq k$, $G[V_i\sqcup V_j]$ is a $M_{ij}$-regular bipartite graph.
    
    \item A \emph{sign matrix} of $M$ is a symmetric integer matrix $M'$ for which $M'_{ij}\equiv M_{ij}\pmod 2$ and $|M'_{ij}|\leq M_{ij}$ for each $1\leq i,j\leq k$.
    
    \item An \emph{$M'$-signing} of a graph lift $G$ of $M$ is an assignment of an element of $\{\pm 1\}$ to each edge of $G$ such that, for each $1\leq i,j\leq k$, the labels of the edges of $G[V_i\sqcup V_j]$ incident to each vertex of $V_i\sqcup V_j$ sum to $M'_{ij}$.
\end{itemize}
\end{definition}

\noindent For example, take
$$M=\begin{pmatrix}0&d\\d&0\end{pmatrix},\ M'=\begin{pmatrix}0&d-2a\\d-2a&0\end{pmatrix}$$
for $a=\lceil 12\sqrt d\rceil$. Then $M'$ is a sign matrix of $M$, and the graph lifts of $M$ are exactly the $d$-regular bipartite graphs. \cref{lem:aug-step} demonstrates the existence of an $M'$-signing of such a graph with top eigenvalue exactly $d-2a$, under certain conditions. We now prove the following lemma, which gives information about the eigendata of graph lifts and matrix-signings.

\begin{lemma}\label{lem:signing-spectrum} Let $(M,M',G)$ be as in \cref{def:sign-matrix}.
\begin{enumerate}[(a)]
    \item The top eigenvalue of $A_G$ is $\lambda_1(M)$; if $v$ is a top eigenvector of $M$, then the vector which assigns each vertex in $V_i$ the value $v_i$ is an eigenvector of $A_G$ with this eigenvalue.
    \item Consider an $M'$-signing $s:E(G)\to\{\pm 1\}$. The eigenvalues of $M'$ are eigenvalues of the adjacency matrix $A_s$ of this signing, and moreover their eigenspaces span the space of vectors $v$ in $\RR^{V(G)}$ which are constant on each part of the $k$-partition.
\end{enumerate}
\end{lemma}
\begin{proof} For (a), it is easy to check that $\lambda_1(M)$ is an eigenvalue of $A_G$ with the eigenvector described as such. The fact that this is the top eigenvalue follows from the observation that, since both $A_G$ and $M$ possess nonnegative entries, $v$ is a Perron eigenvector, and so $\lambda_1(M)$ is the Perron eigenvalue of $A_G$. For (b), it is easy to check that if $w$ is an eigenvector of $M'$ with eigenvalue $\mu$, the vector which assigns a vertex in $V_i$ the value $w_i$ is an eigenvector of $A_s$ with eigenvalue $\mu$ as well.
\end{proof}

\noindent Due to this lemma, the eigenvalue for which we are able to guarantee large multiplicity is $\lambda_1(M')$. We know that it will appear in any $M'$-signing; all that remains is to show that, under certain conditions, no eigenvalue is larger. We first make these conditions explicit.

\begin{definition}\label{def:matrix-set} Denote by $\mathcal M$ the set of triples $(M,M',\beta)$ for which
\begin{enumerate}[(a)]
    \item $M$ is an irreducible $k\times k$ nonnegative integer symmetric matrix with zeros on the diagonal, and $M'$ is a sign matrix of $M$,
    \item $\beta\in(0,1/2)$, and
    \item if $D=(M-M')/2$, $\gamma=\lambda_1(|D-\beta M|)$ (where the absolute value is taken entrywise), and $\sigma=\sum_{i,j}\sqrt{M_{ij}}$, then
    $$\lambda_1\begin{pmatrix}(1-2\beta)\lambda_1(M')+2\gamma+7\sigma&2\lambda_1(D)\\2\lambda_1(D)&2\lambda_1(D)+\sigma\end{pmatrix}\leq\lambda_1(M').$$
\end{enumerate}
\end{definition}

\noindent We may now state the following proposition which generalizes \cref{prop:big-aug-step}.

\begin{proposition}\label{prop:big-aug-step-gen} Suppose $(M,M',\beta)\in\mathcal M$. Let $G$ be a graph lift of $M$ on $n$ vertices for which $\lambda_2(A_G)\leq \lambda_1(M')$, and let $\ell$ be the multiplicity of $\lambda_1(M')$ as an eigenvalue of $A_G$. Then there exists a graph lift $\tilde G$ of $G$ on at most $4n^9$ vertices with second eigenvalue $\lambda_1(M')$ of multiplicity at least $\ell+1$.
\end{proposition}

\noindent As in the proof of \cref{prop:big-aug-step}, the main technical piece will be a lemma which gives the existence of a multiplicity-incrementing lift. We apply such a lift after applying enough Ramanujan lifts between each pair of parts.

\begin{lemma}\label{lem:aug-step-gen} Suppose $(M,M',\beta)\in\mathcal M$, and let $G$ be a graph lift of $M$ on $n$ vertices. Let $m$ be a positive integer. Define $\sigma=\sum_{i,j}\sqrt{M_{ij}}$, as in \cref{def:matrix-set}. Suppose
\begin{enumerate}[(a)]
    \item $\lambda_2(A_G)\leq \lambda_1(M')$,\label{cond:lambda2}
    
    \item at most $m$ eigenvalues of $G$ exceed $\sigma$,\label{cond:not-many-eigs}
    
    \item $A_G$ possesses an orthonormal eigenbasis, a subset of which is a basis for those vectors constant on each part of the $k$-partition of $V(G)$ afforded by $M$, in which each component of each eigenvector with eigenvalue greater than $\sigma$ is at most $\sqrt{m/n}$ in magnitude, and\label{cond:nice-eigs}
    
    \item $n>m^9$.\label{cond:large-graph}
\end{enumerate}
Then there exists a $M'$-signing $s:E(G)\to\{\pm 1\}$ with adjacency matrix $A_s$ satisfying $\lambda_1(A_s)=\lambda_1(M')$.
\end{lemma}

Each of the conditions in this lemma parallels a condition in \cref{lem:aug-step}. In particular, the condition that $(M,M',\beta)\in\mathcal M$ analogizes the degree bound $d\geq 22000$ and the selection of $a=\lceil 12\sqrt d\rceil$. The proof will be quite similar to that of \cref{lem:aug-step}.

\begin{proof} As in \cref{def:matrix-set}, define $D=(M-M')/2$ and $\gamma=\lambda_1(|D-\beta M|)$. Let $V_1\sqcup V_2\sqcup \cdots \sqcup V_k$ be the $k$-partition of $V(G)$ afforded by $M$. Let $U_0\subset\RR^{V(G)}$ be the space of vectors which are constant on each $V_i$. By \cref{lem:signing-spectrum}, the top eigenvalue of $G$ is $\lambda_1(M)$, and its eigenvector lies in $U_0$. Also by \cref{lem:signing-spectrum}, we know that for any $M'$-signing $s$ of $G$ there exists a basis of $U_0$ which is contained within an eigenbasis of $A_s$, one vector in this basis contributes an eigenvalue $\lambda_1(M')$ to $A_s$, and no vector in this basis contributes a larger eigenvalue. So, it suffices to find an $M'$-signing $s$ for which, for any vector $w\in\RR^{V(G)}$ orthogonal to $U_0$, $w^\intercal A_sw\leq \lambda_1(M')$.

\vspace{2mm}

We construct this signing separately on the edges between $V_i$ and $V_j$ for each pair $(i,j)$. Let $G_{ij}$ be the graph $G[V_i\sqcup V_j]$ on vertex set $V(G)$, so that
$$\sum_{1\leq i<j\leq k}A_{G_{ij}}=A_G.$$
Let $v_2,\dots,v_{m'}$ be the eigenvectors of $A_G$ in the partition described in \ref{cond:nice-eigs} with eigenvalues exceeding $\sigma$ and which do not lie in $U_0$, and let $U_1$ be the subspace of $\RR^{V(G)}$ they span. 

\vspace{2mm}

By \ref{cond:not-many-eigs}, $\dim U_1\leq m'<m$. The space $U_1$ satisfies the condition of \cref{cor:factor-selection} with $K=m$. So, since $n\geq m^9$ (condition \ref{cond:large-graph}), we can for each $1\leq i<j\leq k$ find a $D_{ij}$-factor $H_{ij}$ of $G_{ij}$ which satisfies
$$\left|v^\intercal\left(A_{H_{ij}}-\frac{D_{ij}}{M_{ij}}A_{G_{ij}}\right)v\right|\leq 7\sqrt{M_{ij}}\|v\|^2$$
for every $v\in U_1$. We conclude, letting $H$ be the graph formed by the union of the edge sets $H_{ij}$ and letting $B=\sum \frac{D_{ij}}{M_{ij}}A_{G_{ij}}$, that
$$\left|v^\intercal(A_H-B)v\right|\leq 7\|v\|^2\left(\sum_{1\leq i<j\leq k}\sqrt{M_{ij}}\right)=\frac 72\sigma\|v\|^2$$
for every $v\in U_1$. Let $s$ be the signing for which edges in $H_{ij}$ for any $(i,j)$ are assigned $-1$ and all other edges are assigned $1$. By definition, $s$ is an $M'$-signing; we also have $A_s=A_G-2A_H$. 

\vspace{2mm}

Now, let $w\in\RR^{V(G)}$ be any unit vector orthogonal to $U_0$; write $w=x+y$, where $x\in U_1$ and $y$ lies in the orthogonal complement of $U_0+U_1$; note that $\|x\|^2+\|y\|^2=1$. We first see
\begin{align*}
w^\intercal A_Hw
&=x^\intercal A_Hx+2x^\intercal A_Hy+y^\intercal A_Hy\\
&\geq x^\intercal A_Hx-2\|x\|\|y\|\|A_H\|-\|y\|^2\|A_H\|\\
&\geq x^\intercal Bx-\frac 72\sigma\|x\|^2-(2\|x\|\|y\|+\|y\|^2)\|A_H\|\\
&\geq \beta x^\intercal A_Gx-\left(\|B-\beta A_G\|+\frac 72\sigma\right)\|x\|^2-(2\|x\|\|y\|+\|y\|^2)\|A_H\|.
\end{align*}
Since $H$ is a graph lift of $D$, and the entries of $D$ are nonnegative integers, \cref{lem:signing-spectrum}(a) implies that $\|A_H\|=\lambda_1(A_H)=\lambda_1(D)$. We now upper-bound $\|B-\beta A_G\|$. For any unit vector $u\in \RR^{V(G)}$, write $u=u_1+\cdots+u_k$, where $u_i$ is supported on $V_i$. We have, similarly to the proof of \cref{lem:signing-spectrum},
\begin{align*}
u^\intercal (B-\beta A_G)u
&=\sum_{1\leq i<j\leq k}\left(\frac{D_{ij}}{M_{ij}}-\beta\right)u^\intercal A_{G_{ij}}u\\
&=\sum_{1\leq i<j\leq k}\left(\frac{D_{ij}}{M_{ij}}-\beta\right)2M_{ij}\|u_i\|\|u_j\|\\
&=\sum_{i=1}^k\sum_{j=1}^k|D_{ij}-\beta M_{ij}|\|u_i\|\|u_j\|\leq \lambda_1(|D-\beta M|)=\gamma,
\end{align*}
since $\|u_1\|^2+\cdots+\|u_k\|^2=1$. So, $\|B-\beta A_G\|\leq \gamma$. This implies that
\begin{align*}
w^\intercal A_sw
&=w^\intercal(A_G-2A_H)w\\
&=x^\intercal A_Gx+y^\intercal A_Gy-2w^\intercal A_Hw\\
&\leq x^\intercal A_Gx+\sigma\|y\|^2-2\beta x^\intercal A_Gx+\left(2\gamma+7\sigma\right)\|x\|^2+4\lambda_1(D)\|x\|\|y\|+2\lambda_1(D)\|y\|^2\\
&=(1-2\beta)x^\intercal A_Gx+\left(2\gamma+7\sigma\right)\|x\|^2+4\lambda_1(D)\|x\|\|y\|+\big(2\lambda_1(D)+\sigma\big)\|y\|^2\\
&\leq\left((1-2\beta)\lambda_1(M')+2\gamma+7\sigma\right)\|x\|^2+4\lambda_1(D)\|x\|\|y\|+\big(2\lambda_1(D)+\sigma\big)\|y\|^2,
\end{align*}
where we have used condition \ref{cond:lambda2} in the last bound, since $x\in U_0$. Since we have for any $R,S,T$ and any $\|x\|,\|y\|\in\RR$ with $\|x\|^2+\|y\|^2=1$, 
$$R\|x\|^2+2S\|x\|\|y\|+T\|y\|^2=\begin{pmatrix}\|x\|&\|y\|\end{pmatrix}\begin{pmatrix}R&S\\S&T\end{pmatrix}\begin{pmatrix}\|x\|\\\|y\|\end{pmatrix}\leq \lambda_1\begin{pmatrix}R&S\\S&T\end{pmatrix},$$
the result follows from the definition of $\mathcal M$. 
\end{proof}

\begin{proof}[Proof of \cref{prop:big-aug-step-gen}] Armed with the above lemma, the proof will be very similar to that of \cref{prop:big-aug-step}. Let $t$ be a nonnegative integer. We will first construct a graph lift $G_t$ of $M$ on $n_t:=n2^t$ vertices for which, for each eigenvalue $\mu$ of $A_{G_t}$ with $\mu\geq \sigma$, the $\mu$-eigenspace of $A_{G_t}$ is exactly
$$\left\{\underbrace{x\oplus x\oplus \cdots\oplus x}_{2^t\text{ copies of }x}\colon A_Gx=\mu x\right\}.$$
Indeed, starting with $G_0=G$ when $t=0$, we can construct these graphs by successively applying Ramanujan $2$-lifts between each pair of parts. Let $V_1\sqcup V_2\sqcup\cdots\sqcup V_k$ be the $k$-partition of $V(G_t)$ afforded by $M$. For each $1\leq i<j\leq k$, let $s_{ij}:E(G_t[V_i\sqcup V_j])\to\{\pm 1\}$ be a signing of the $M_{ij}$-regular bipartite graph $G_t[V_i\sqcup V_j]$ whose adjacency matrix has second eigenvalue at most $2\sqrt{M_{ij}-1}<2\sqrt{M_{ij}}$, guaranteed to exist by \cref{thm:ramanujan} (if $M_{ij}=0$, then $s_{ij}$ is trivial and its adjacency matrix has second eigenvalue $0=2\sqrt{M_{ij}}$). Construct a signing $s:E(G_t)\to\{\pm 1\}$ by combining the partial signings $s_{ij}$, so that $A_s=\sum_{1\leq i<j\leq k}A_{s_{ij}}$ has top eigenvalue at most
$$\sum_{1\leq i<j\leq k}\lambda_1(A_{s_{ij}})\leq \sum_{1\leq i<j\leq k}2\sqrt{M_{ij}}=\sum_{i=1}^k\sum_{j=1}^k\sqrt{M_{ij}}=\sigma.$$
Then the $2$-lift $G_{t+1}$ of $G_t$ associated to $s$ has no eigenvalues exceeding $\sigma$ which do not come from $A_{G_t}$, and so $G_{t+1}$ satisfies the desired properties. In particular, for each $t$,
$$\lambda_2(A_{G_t})\leq\max(\lambda_2(A_G),\sigma)\leq \lambda_1(M'),$$
(we know $\sigma\leq\lambda_1(M')$ by condition (c) in the definition of $\mathcal M$), and $A_{G_t}$ has eigenvalue $\lambda_1(M')$ with multiplicity $k$. Identically to the proof of \cref{prop:big-aug-step}, conditions (b) and (c) of \cref{lem:aug-step-gen} are satisfied by $G_t$, since we have described the eigenspaces of $A_{G_t}$ with large eigenvalue. So, we can apply this lemma to get an $M'$-signing $s:E(G)\to\{\pm 1\}$ with $\lambda_1(A_s)=\lambda_1(M')$ as long as $2^j>n^8$. The spectrum of the associated $2$-lift contains the eigenvalue $\lambda_1(M')$ with multiplicity at least $\ell+1$, and so selecting $t$ to be the smallest positive integer satisfying $2^t>n^8$ finishes the proof.
\end{proof}

These multiplicity-incrementing steps give us graphs with large multiplicity of fixed second eigenvalue.

\begin{corollary}\label{cor:general-loglog} Suppose $(M,M',\beta)\in \mathcal M$ and that there exists a connected graph lift $G_0$ of $M$ on at most $r$ vertices with $\lambda_2(A_{G_0})\leq \lambda_1(M')$. Then there exists an infinite sequence $r=n_0<n_1<\cdots$ of positive integers, each satisfying $n_{i+1}\leq 4n_i^9$, for which there exists a graph on $n_i$ vertices with top eigenvalue $\lambda_1(M)$ and second eigenvalue $\lambda_1(M')$ with multiplicity at least $i=\Omega(\log\log n_i)$. Furthermore, if the all-ones vector is an eigenvector of $M$, then these graphs can be chosen to be regular.    
\end{corollary}

\begin{proof} This parallels the proof of \cref{thm:large-mult-nice}(a). In the previous section. Let $n_0=r$, and define a sequence $(n_i)$ recursively so that $n_{i+1}$ is the unique integer in the interval $(2n_i^9,4n_i^9]$ for which $n_{i+1}/n_i$ is a power of $2$. Repeated applications of \cref{prop:big-aug-step-gen} show that, for each $i\geq 1$, there exists a graph lift $G_i$ of $M$ such that
\begin{enumerate}[(i)]
    \item $G_i$ has $n_i$ vertices, and
    \item $\lambda_1(M')$ is the second eigenvalue of $G_i$ and has multiplicity at least $i$.
\end{enumerate}
Since $G_i$ is a graph lift of $M$, its top eigenvalue is $\lambda_1(M)$ by \cref{lem:signing-spectrum}(a). If the all-ones vector is an eigenvector of $M$, then it is the Perron eigenvector of $A_{G_i}$ (again by \cref{lem:signing-spectrum}(a)) for each $i$, and so (since each $G_i$ will be connected) each $G_i$ is regular with degree $\lambda_1(M)$.
\end{proof}

To conclude the proof of \cref{thm:large-mult-nice}(b), we need to choose our matrices $M$ and $M'$ and show that they satisfy the necessary properties. We do this in the following lemma.

\begin{lemma}\label{lem:matrix-choice} Let $t>u$ be positive integers of the same parity. Define the matrices
$$M=\begin{pmatrix}&t&t&3\\t&&3&t\\t&3&&t\\3&t&t&\end{pmatrix},\ M'=\begin{pmatrix}&u&u&1\\u&&-3&u\\u&-3&&u\\1&u&u&\end{pmatrix},$$
and let $\beta=(t-u)/(2t)$. Suppose $(5/6)t\leq u\leq t-56\sqrt t-62$. Then
\begin{enumerate}[(a)]
    \item $\lambda_1(M')=2\sqrt{u^2+1}-1$,
    \item there exists a (connected) graph lift $G_0$ of $M$ on $4t$ vertices with $\lambda_2(A_{G_0})\leq\lambda_1(M')$, and
    \item $(M,M',\beta)\in\mathcal M$.
\end{enumerate}
\end{lemma}
\begin{proof} We first note that the bounds on $u$ imply that $\beta\leq 1/12$ and $t,u\geq 10^4$. To prove part (a), it suffices to compute the eigenvalues of $M'$ explicitly to be
$$\left\{2\sqrt{u^2+1}-1,2\sqrt{u^2+1}-1,3,-1\right\}.$$
Since $u\geq 2$, the largest of these is $2\sqrt{u^2+1}-1$.

\vspace{2mm}

To prove part (b), we construct such a graph. Between the parts to be connected by $t$-regular bipartite graphs, we place complete bipartite graphs, forming a graph $G_0^{(1)}$. Then we place $3$-regular bipartite graphs between the other two pairs of parts arbitrarily (this is possible since $t\geq 3$); let the union of these two $3$-regular bipartite graphs be $G_0^{(2)}$. Since $G_0^{(2)}$ is $3$-regular, its adjacency matrix has spectral norm at most $3$. Weyl's inequality thus tells us
$$\lambda_2(A_{G_0})\leq \lambda_2\left(A_{G_0^{(1)}}\right)+3.$$
The spectrum of $G_0^{(1)}$ consists of one copy each of $\pm 2t$ and $4t-2$ copies of $0$. This shows $\lambda_2(A_{G_0})\leq 3\leq\lambda_1(M')$, as desired.

\vspace{2mm}

Finally, we show (c). It is clear that $M'$ is a sign matrix of $M$; the rest requires some computation. We have the matrices
$$D:=\frac{M-M'}2=\begin{pmatrix}&t\beta&t\beta&1\\t\beta&&3&t\beta\\t\beta&3&&t\beta\\1&t\beta&t\beta&\end{pmatrix},\ D-\beta M=\begin{pmatrix}&&&1-3\beta\\&&3-3\beta&\\&3-3\beta&&\\1-3\beta&&&\end{pmatrix};$$
this gives $\lambda_1(D)\leq 2t\beta+3$ and $\gamma=\lambda_1(|D-\beta M|)\leq 3$. Also, $\sigma=8\sqrt t+4\sqrt 3<8\sqrt t+7$. Since the top eigenvalue of positive matrices may only increase if entries are increased, it suffices to show, using that $2u-1\leq \lambda_1(M')\leq 2u$,
$$L:=\lambda_1\begin{pmatrix}(1-2\beta)(2u)+56\sqrt t+55&4t\beta+6\\4t\beta+6&4t\beta+8\sqrt t+13    
\end{pmatrix}\leq 2u-1.$$
By sub-additivity of top eigenvalue, we have
$$L\leq \lambda_1\begin{pmatrix}56\sqrt t+55&6\\6&8\sqrt t+13\end{pmatrix}+t\cdot \lambda_1\begin{pmatrix}2(1-2\beta)^2&4\beta\\4\beta&4\beta\end{pmatrix}\leq 56\sqrt t+61+t(2-6\beta),$$
where we have used the computation that
$$\lambda_1\begin{pmatrix}2(1-2\beta)^2&4\beta\\4\beta&4\beta\end{pmatrix}\leq 2-6\beta$$
for $\beta\leq 1/12$. We thus need
$$56\sqrt t+62\leq 2\beta t=t-u,$$
which holds by assumption.
\end{proof}

We conclude with the proof of \cref{thm:large-mult-nice}(b). In fact, we following prove the (ever so slightly) stronger statement; we will need its precision to show \cref{thm:main}.

\begin{theorem}\label{thm:large-mult-nonint} Let $u\geq 10^5$ be an integer. There exists some integer $d\leq 5u/2$ for which, for an infinite sequence $n_0<n_1<\cdots$ of positive integers, each satisfying $n_{i+1}\leq 4n_i^9$, there exists a $d$-regular graph on $n_i$ vertices with second eigenvalue exactly $2\sqrt{u^2+1}-1$ of multiplicity $i=\Omega(\log\log n_i)$.
\end{theorem}
\begin{proof} Let $t$ be a positive integer of the same parity as $u$ satisfying $5t/6\leq u\leq t-56\sqrt t-62$, which exists since $u\geq 10^5$. Define
$$M=\begin{pmatrix}&t&t&3\\t&&3&t\\t&3&&t\\3&t&t&\end{pmatrix},\ M'=\begin{pmatrix}&u&u&1\\u&&-3&u\\u&-3&&u\\1&u&u&\end{pmatrix},$$
and $\beta=(t-u)/(2t)$. By \cref{lem:matrix-choice}, $\lambda_1(M')=2\sqrt{u^2+1}-1$, there exists a connected graph lift $G_0$ of $M$ on $4t$ vertices with $\lambda_2(A_{G_0})\leq \lambda_1(M')$, and $(M,M',\beta)\in \mathcal M$. The fact that the all-ones vector is an eigenvector of $M$ with eigenvalue $2t+3\leq 12u/5+3\leq 5u/2$ is enough to finish the proof by appealing to \cref{cor:general-loglog}.
\end{proof}

\section{Application: Equiangular lines}\label{sec:equi}

We apply \cref{thm:large-mult-nonint} to the problem of equiangular lines with fixed angle. Recall that, if the spectral radius order $k(\lambda)$ of $\lambda=(1-\alpha)/(2\alpha)$ is finite, then $N_\alpha(n)\geq (1+\varepsilon)n$ for some $\varepsilon>0$ and large $n$. So, to prove \cref{thm:main}, we need to find $\lambda$ which is not the top eigenvalue of any graph, but which may still be the second eigenvalue. To this end, we begin by giving a necessary condition for $k(\lambda)<\infty$. This condition is stated in \cite{JiangPolyanskii}.

\begin{lemma}\label{lem:top-eig-char} Suppose $\lambda$ is such that $k(\lambda)<\infty$. Then $\lambda$ is an algebraic integer all of whose Galois conjugates are real and at most $\lambda$ in magnitude.
\end{lemma}
\begin{proof} Let $G$ be any graph, and consider the characteristic polynomial $p_G(t)=\det(tI-A_G)$. This polynomial is monic of degree $|V(G)|$ and has integer coefficients; this implies (1) that every root of $p_G$ is an algebraic integer and (2) that if $\lambda$ is a root of $p_G$, then every Galois conjugate of $\lambda$ is as well (the minimal polynomial of $\lambda$ must divide $p_G$). Since $A_G$ is symmetric, every eigenvalue of $A_G$ is real. The remainder of the lemma follows from the fact that, since every entry of $A_G$ is nonnegative, every eigenvalue of $A_G$ is at most $\lambda_1(A_G)$ in magnitude, since $\lambda_1(A_G)$ is a Perron eigenvalue.
\end{proof}

\begin{remark}\label{rmk:top-eig-char} This condition is not sufficient. For any interval $I\subset\RR$ of length strictly exceeding $4$ (resp. strictly less than $4$), there exist infinitely many (resp. finitely many) totally real algebraic integers all of whose Galois conjugates lie within $I$ \cite{Robinson}. In particular, there are infinitely many totally real algebraic integers all of whose conjugates lie in $[-1.995,2.006]$, and, for all but finitely many of them, the largest conjugate strictly exceeds $2$. On the other hand, due to a result of \cite{CDG} (see \cite[Theorem 2.3]{CvetkovicRowlinson} and surrounding remarks) no undirected graph has largest eigenvalue in the interval $(2,2.007]$.
\end{remark}

\noindent Our choice of $\lambda$ is described in the following corollary.

\begin{corollary}\label{cor:non-top-eig} If $t$ is a positive integer which is not a square, then $k(2\sqrt t-1)=\infty$.
\end{corollary}
\begin{proof} For non-square $t$, $-2\sqrt t-1$ is a Galois conjugate of $2\sqrt t-1$ and exceeds $2\sqrt t-1$ in magnitude.
\end{proof}

Now, we give a result which allows us to derive sets of equiangular lines in every dimension from the graphs we have constructed.

\begin{lemma}\label{lem:regular-to-lines} Suppose there exists an $n$-vertex $d$-regular graph $G$ with second eigenvalue $\lambda>0$ of multiplicity $k$, and let $\alpha=1/(2\lambda+1)$. 
\begin{enumerate}[(a)]
    \item If $n\geq 2d$, then $N_\alpha(n-k)\geq n$. 
    \item If moreover $\lambda\geq 2\sqrt{d-1}$ and $\lambda\geq 2d/3$, then $N_\alpha(m)\geq m+k$ for every integer $m\geq n-k$.
\end{enumerate}
\end{lemma}
\begin{proof} Let $A_G$ be the adjacency matrix of $G$; note that, since $G$ is $d$-regular, the vector $v_1$ with each component $1/\sqrt n$ is the top eigenvector of $A_G$ with eigenvalue $d$. Now, let $J$ be the $n\times n$ all-ones matrix, and consider the $n\times n$ symmetric matrix
$$M=(1-\alpha)I+\alpha J-2\alpha A_G.$$
Since $J=nv_1v_1^\intercal$ and $A_G$ are simultaneously diagonalizable, the spectrum of $M$ consists of one copy of $1-\alpha+\alpha n-2\alpha d>0$ corresponding to the all-ones vector and a copy of $(1-\alpha)-2\alpha\mu$ for every eigenvalue $\mu<d$ of $A_G$, with multiplicity equal to that of $\mu$. Since $\lambda_2(A_G)=\lambda=(1-\alpha)/(2\alpha)$, this means that $M$ is positive semidefinite and has eigenvalue $0$ with multiplicity $k$. So, $M$ is a Gram matrix of $n$ vectors $w_1,\dots,w_n\in \RR^{n-k}$. By its definition, every diagonal entry of $M$ is $1$ and every off-diagonal entry is $\pm\alpha$. So, each $w_i$ is a unit vector, and the lines spanned by distinct $w_i$ each meet at an angle of $\arccos\alpha$. This finishes the proof of part (a).

\vspace{2mm}

We now show part (b). By the above argument, it suffices to show, for each integer $\ell\geq n$, the existence of a positive semidefinite $\ell\times \ell$ matrix whose diagonal consists only of ones, whose off-diagonal entries are each $\pm\alpha$, and whose rank is at most $\ell-k$. We will show that, if an $n$-vertex graph $G$ exists as in the lemma statement, such a matrix exists for each $\ell\in [n,2n)$. From here, it suffices to note that, if such an $n$-vertex graph $G=G_0$ exists, there is a $2^in$-vertex graph $G_i$ also satisfying the necessary properties for all $i\geq 0$: by \cref{thm:ramanujan} and the fact that $\lambda\geq 2\sqrt{d-1}$, we can pick $G_{i+1}$ to be a suitable $2$-lift of $G_i$.

\vspace{2mm}

Now, consider the $n\times n$ matrix $M=(1-\alpha)I+\alpha J-2\alpha A_G$ described above; it is positive semidefinite, and the all-ones vector is an eigenvector of $M$ with eigenvalue $\alpha n-(2\alpha d-1+\alpha)$. For each integer $\ell\in (n,2n)$, define the $\ell\times \ell$ matrix $M_\ell$ by appending $\ell-n$ rows and columns to $M$, filling the diagonal with ones and all off-diagonal entries with $\alpha$. If $Mv=0$, then $v\oplus (0,\dots,0)$ is an eigenvector of $M_\ell$ with eigenvalue $0$, so $M_\ell$ has rank at most $\ell-k$. Moreover, the diagonal entries of $M_\ell$ are all $1$ and the off-diagonal entries are $\pm\alpha$. So, it suffices to prove that $M_\ell$ is positive semidefinite. For each $r$, let $1_r$ be the all-ones vector of length $r$, and let $t=\ell-n$; we need to show
\begin{align*}
\begin{pmatrix}u^\intercal&v^\intercal\end{pmatrix}&\begin{pmatrix}M&\alpha 1_n1_t^\intercal \\\alpha 1_t1_n^\intercal&\alpha 1_t1_t^\intercal+(1-\alpha)I_t\end{pmatrix}\begin{pmatrix}u\\v\end{pmatrix}\\
&=u^\intercal Mu+2\alpha(u^\intercal 1_n)(v^\intercal 1_t)+\alpha(v^\intercal 1_t)^2+(1-\alpha)v^\intercal v\geq 0
\end{align*}
for each $u\in\RR^n$ and $v\in\RR^t$. Indeed, since $1_n$ is an eigenvector of $M$ with eigenvalue $1-\alpha+\alpha n-2\alpha d$,
\begin{align*}
(u\oplus v)^\intercal M_\ell(u\oplus v)&=u^\intercal Mu+2\alpha(u^\intercal 1_n)(v^\intercal 1_t)+\alpha(v^\intercal 1_t)^2+(1-\alpha)v^\intercal v\\
&\geq \left(\alpha-\frac{2\alpha d-(1-\alpha)}n\right)(u^\intercal 1_n)^2+2\alpha xy+\alpha y^2+\left(\frac{1-\alpha}t\right)(v^\intercal 1_t)^2\\
&=\left(\alpha-\frac{2\alpha d-(1-\alpha)}n\right)x^2+2\alpha xy+\left(\alpha+\frac{1-\alpha}t\right)y^2,
\end{align*}
where we have set $x=u^\intercal 1_n$ and $y=v^\intercal 1_t$, and used that $M$ is positive semidefinite. It thus suffices that
$$\left(\alpha-\frac{2\alpha d-(1-\alpha)}n\right)\left(\alpha+\frac{1-\alpha}t\right)\geq \alpha^2.$$
The left side is decreasing in $t$, so we need only verify the inequality when $t=n$. Upon substituting $(1-\alpha)/\alpha=2\lambda$, the sufficient condition becomes
$$\left(1-\frac{2(d-\lambda)}n\right)\left(1+\frac{2\lambda}n\right)\geq 1.$$
This simplifies to $\frac\lambda{d-\lambda}-1\geq \frac{2\lambda}n$; the facts that $\lambda\geq 2d/3$ and $n\geq 2d\geq 2\lambda$ suffice to verify this inequality.
\end{proof}

\begin{remark}\label{rmk:non-augment} It is not generally true that $N_\alpha(n+1)\geq N_\alpha(n)+1$. In fact, $N_{1/3}(7)=\cdots=N_{1/3}(14)=28$ \cite{GSY}.
\end{remark}

Finally, we prove \cref{thm:main}.

\begin{proof}[Proof of \cref{thm:main}] Let $u\geq 10^5$ be an integer, and let
$$\alpha=\frac1{4\sqrt{u^2+1}-1}.$$
We will show that $n+\Omega(\log\log n)\leq N_\alpha(n)\leq n+O(n/\log\log n)$. 

\vspace{2mm}

Let $\lambda=(1-\alpha)/(2\alpha)=2\sqrt{u^2+1}-1$. The upper bound on $N_\alpha(n)$ follows from \cref{thm:JTYZZ} and \cref{cor:non-top-eig}: since $u^2+1$ is not a square, \cref{cor:non-top-eig} gives that $k(\lambda)=\infty$, so \cref{thm:JTYZZ}(b) implies $N_\alpha(n)=n+O(n/\log\log n)$. For the lower bound, we use the graphs from \cref{thm:large-mult-nonint}. For some $d\leq 5u/2\leq 3\lambda/2$ and some sequence $n_0<n_1<\cdots$ satisfying $n_i\leq 4n_{i-1}^9$, there exist $d$-regular graphs on $n_i$ vertices with second eigenvalue $\lambda$ of multiplicity $\Omega(\log\log n_i)$. Defining $f:\mathbb N\to\RR$ by $f(n)=0$ for $n<n_0$ and $f(n)=i$ whenever $n_i\leq n<n_{i+1}$, we have $f(n)=\Omega(\log\log n)$. Since $\lambda\geq 2d/3$, \cref{lem:regular-to-lines} gives that $N_\alpha(n)\geq n+i$ as long as $n\geq n_i-i$; this implies
$$N_\alpha(n)\geq n+f(n)\geq n+\Omega(\log\log n),$$
as desired.
\end{proof}

\section*{Acknowledgment}

The author would like to thank Yufei Zhao for his mentorship, for suggesting the problem, and for providing helpful comments and suggestions. 

\bibliographystyle{halpha}
\bibliography{bib.bib}

\newcommand{\etalchar}[1]{$^{#1}$}
\begin{thebibliography}{JTY{\etalchar{+}}21}
\expandafter\ifx\csname url\endcsname\relax
  \def\url#1{\texttt{#1}}\fi
\expandafter\ifx\csname doi\endcsname\relax
  \def\doi#1{\burlalt{doi:#1}{http://dx.doi.org/#1}}\fi
\expandafter\ifx\csname urlprefix\endcsname\relax\def\urlprefix{URL }\fi
\expandafter\ifx\csname href\endcsname\relax
  \def\href#1#2{#2}\fi
\expandafter\ifx\csname burlalt\endcsname\relax
  \def\burlalt#1#2{\href{#2}{#1}}\fi

\bibitem[BL06]{BiluLinial}
Yonatan Bilu and Nathan Linial.
\newblock Lifts, {Discrepancy} and {Nearly} {Optimal} {Spectral} {Gap}.
\newblock {\em Combinatorica}, 26(5):495--519, Oct. 2006.
\newblock \doi{10.1007/s00493-006-0029-7}.

\bibitem[CDG82]{CDG}
Dragoš Cvetković, Michael Doob, and Ivan Gutman.
\newblock On graphs whose spectral radius does not exceed $(2+\sqrt5)^{1/2}$.
\newblock {\em Ars Combinatoria}, 14:225--239, 1982.

\bibitem[CR90]{CvetkovicRowlinson}
D.~Cvetković and P.~Rowlinson.
\newblock The largest eigenvalue of a graph: {A} survey.
\newblock {\em Linear and Multilinear Algebra}, 28(1-2):3--33, Oct. 1990.
\newblock \doi{10.1080/03081089008818026}.

\bibitem[dC00]{deCaen}
D.~de~Caen.
\newblock Large {Equiangular} {Sets} of {Lines} in {Euclidean} {Space}.
\newblock {\em The Electronic Journal of Combinatorics}, 7:R55, Nov. 2000.
\newblock \doi{10.37236/1533}.

\bibitem[GSY21]{GSY}
Gary R.~W. Greaves, Jeven Syatriadi, and Pavlo Yatsyna.
\newblock Equiangular {Lines} in {Low} {Dimensional} {Euclidean} {Spaces}.
\newblock {\em Combinatorica}, 41(6):839--872, Dec. 2021.
\newblock \doi{10.1007/s00493-020-4523-0}.

\bibitem[HSZZ22]{HSZZ}
Milan Haiman, Carl Schildkraut, Shengtong Zhang, and Yufei Zhao.
\newblock Graphs with high second eigenvalue multiplicity.
\newblock {\em Bulletin of the London Mathematical Society}, 54(5):1630--1652,
  2022.
\newblock \doi{10.1112/blms.12647}.

\bibitem[JP20]{JiangPolyanskii}
Zilin Jiang and Alexandr Polyanskii.
\newblock Forbidden {Subgraphs} for {Graphs} of {Bounded} {Spectral} {Radius},
  with {Applications} to {Equiangular} {Lines}.
\newblock {\em Israel Journal of Mathematics}, 236(1):393--421, Mar. 2020.
\newblock \doi{10.1007/s11856-020-1983-2}.

\bibitem[JTY{\etalchar{+}}21]{JTYZZ}
Zilin Jiang, Jonathan Tidor, Yuan Yao, Shengtong Zhang, and Yufei Zhao.
\newblock Equiangular lines with a fixed angle.
\newblock {\em Annals of Mathematics}, 194(3):729--743, Nov. 2021.
\newblock \doi{10.4007/annals.2021.194.3.3}.

\bibitem[LS73]{LemmensSeidel}
P.~W.~H Lemmens and J.~J Seidel.
\newblock Equiangular lines.
\newblock {\em Journal of Algebra}, 24(3):494--512, Mar. 1973.
\newblock \doi{10.1016/0021-8693(73)90123-3}.

\bibitem[MRS21]{MRS}
Theo McKenzie, Peter Michael~Reichstein Rasmussen, and Nikhil Srivastava.
\newblock Support of closed walks and second eigenvalue multiplicity of graphs.
\newblock In {\em Proceedings of the 53rd {Annual} {ACM} {SIGACT} {Symposium}
  on {Theory} of {Computing}}, {STOC} 2021, pages 396--407, New York, NY, USA,
  June 2021. Association for Computing Machinery.
\newblock \doi{10.1145/3406325.3451129}.

\bibitem[MSS15]{MSS}
Adam Marcus, Daniel Spielman, and Nikhil Srivastava.
\newblock Interlacing families {I}: {Bipartite} {Ramanujan} graphs of all
  degrees.
\newblock {\em Annals of Mathematics}, pages 307--325, July 2015.
\newblock \doi{10.4007/annals.2015.182.1.7}.

\bibitem[Rob62]{Robinson}
Raphael~M. Robinson.
\newblock Intervals containing infinitely many sets of conjugate algebraic
  integers.
\newblock In {\em Studies in mathematical analysis and related topics}, pages
  305--315. Stanford Univ. Press, Stanford, Calif., 1962.

\end{thebibliography}
\end{document}